\documentclass[12pt,reqno,a4paper]{amsart}
\usepackage{blindtext}
\usepackage{fullpage}
\usepackage{mathtools}
\usepackage{longtable}
\usepackage{amsmath,amssymb,amsthm}
\usepackage{amscd}
\usepackage{bm}
\usepackage{hyperref}   
\usepackage{dsfont}
\usepackage{enumerate}
\usepackage{epsfig}
\usepackage{float, graphicx}
\usepackage{latexsym, amsxtra}
\usepackage{mathrsfs}
\usepackage{multicol}
\usepackage[normalem]{ulem}
\usepackage{psfrag}
\usepackage[parfill]{parskip}
\usepackage{stmaryrd}
\usepackage{tikz}
\usepackage[T1]{fontenc}
\usepackage{url}
\usepackage{verbatim}
\usepackage{indentfirst}
\usepackage{tikz-cd}
\usepackage{booktabs}
\usepackage{svg}

\usepackage{mathtools}

\usepackage{caption} 

\makeatletter
\def\thm@space@setup{%
	\thm@preskip=2ex \thm@postskip=2ex
}
\makeatother

\hypersetup{hidelinks}

\newtheorem{thm}{Theorem~}[section]
\newtheorem{lem}[thm]{Lemma~}

\newtheorem{cor}[thm]{Corollary~}

\theoremstyle{remark}
\newtheorem{rmk}[thm]{Remark~}

\theoremstyle{definition}  % 设置环境样式为 "definition"

\newcommand{\CC}{\mathbb{C}}
\newcommand{\ZZ}{\mathbb{Z}}
\newcommand{\RR}{\mathbb{R}}
\newcommand{\LL}{\mathbb{L}}
\newcommand{\PP}{\mathbb{P}}

\newcommand{\calM}{\mathcal{M}}

\newcommand\PU{\mathrm{PU}}

\newcommand{\bs}{\backslash}

\newcommand{\pFq}[5]{{}_{#1}F_{#2}\left(\genfrac..{0pt}{}{#3}{#4};#5\right)}

\newcommand{\Fd}[5]{{}F^{#1}_{#2}\left(\genfrac..{0pt}{}{#3}{#4};#5\right)}

\title{Double integrals and transformation formulas for Appell--Lauricella hypergeometric functions $F_D$}
\vspace{1.2cm}
\author{Shihao Wang, Chenglong Yu, Zhiwei Zheng}
\date{}

\newcommand{\Addresses}{{% additional braces for segregating \footnotesize
		\bigskip
		\footnotesize

        S.~Wang, \textsc{Tsinghua University, Beijing, China}\par\nopagebreak
		\textit{Email address}: \texttt{wangshihao112@gmail.com}
        
		C.~Yu, \textsc{Center for Mathematics and Interdisciplinary Sciences, Fudan University and
Shanghai Institute for Mathematics and Interdisciplinary Sciences (SIMIS), Shanghai, China}\par\nopagebreak
		\textit{Email address}: \texttt{yuchenglong@simis.cn}
		
		\medskip
		
		Z.~Zheng, \textsc{Tsinghua University, Beijing, China}\par\nopagebreak
		\textit{Email address}: \texttt{zhengzhiwei@mail.tsinghua.edu.cn}
}}

\begin{document}
\bibliographystyle{amsalpha}

\begin{abstract}
The monodromy of hypergeometric functions can govern the properties of the functions themselves. Previously, the second and third authors studied the commensurability relations among monodromy groups of the Appell--Lauricella hypergeometric functions using Deligne--Mostow theory and the geometric correspondence between curves and surfaces. In this paper, we apply the same construction to obtain transformation formulas among these hypergeometric functions. This also provides an alternative approach to some of Goursat's quadratic transformations via double integrals and Fubini's theorem.
\end{abstract}
	
	\maketitle
 \setcounter{tocdepth}{1}
	\tableofcontents

 \section{Introduction}
The quadratic or higher order transformations of Euler--Gauss hypergeometric functions date back to Gauss \cite{gauss1866determinatio}, Kummer \cite{Kummer1836} and Goursat \cite{goursat1881equation}. They obtained these formulas by studying the corresponding hypergeometric differential equations. The work of Riemann \cite{riemann1857beitrage}, Fuchs \cite{fuchs1866}, Schwarz \cite{schwarz1873} and many other mathematicians established the powerful connection between monodromy groups of the hypergeometric systems and properties of hypergeometric functions. Especially when the two monodromy groups are commensurable (namely, after suitable conjugation, they have common finite-index subgroup), the corresponding hypergeometric functions are expected to be related by transformation formulas. 

For multivariable generalization of hypergeometric functions, such as Appell--Lauricella hypergeometric functions, there are also many works for transformation formulas following similar philosophy of comparing the differential systems, see for example \cite{koike2007isogeny,koike2008extended,matsumoto2009some,matsumoto2010transformation,otsubo2021new}. For the monodromy groups of Appell--Lauricella hypergeometric functions, Sauter \cite{sauter1990isomorphisms} and Deligne--Mostow \cite{deligne1993commensurabilities} studied their commensurability relations as subgroups in $\PU(1, 2)$. Deligne--Mostow also related them to identities of hypergeometric functions, see \cite[Chapter 13]{deligne1993commensurabilities}. The method is based on the comparison of the moduli spaces with conic complex hyperbolic structures. In \cite{yu2024comm}, the second and third authors studied the commensurability relations among monodromy groups of Appell--Lauricella hypergeometric functions in Deligne--Mostow theory through the geometry of certain surfaces with double fibrations. There are new relations found, see \cite[Table 2]{yu2024comm} for a full list of such commensurability relations.

In this paper, we apply the same construction (surfaces with double fibrations) to obtain transformation formulas between the corresponding hypergeometric functions. The key idea is to study period integrals for certain volume forms on these surfaces and use two different iterated integrals reducing to hypergeometric functions. This can be seen as a period integral version of the Hodge isometry obtained in \cite[Theorem 3.6]{yu2024comm}. 

We list some of the transformations here. For suitable range of parameters, we have
\begin{enumerate}
    \item Theorem \ref{thm:transformula(1,2)(1,1)(1,0)}:   \begin{eqnarray*}
        &&\Fd{(3)}{D}{2-2a;1-a,1-a,1-a}{4-4a}{\frac{2(\beta_{2}-\beta_{1})t}{\alpha_{2}\beta_{1}-\alpha_{1}\beta_{2}},-\frac{2t}{\alpha_{1}},-\frac{2t}{\alpha_{2}}}\\
        &=&C\cdot \Fd{(3)}{D}{1-a; a-\frac{1}{2},1-a,1-a}{\frac{5}{2}-2a}{\frac{1}{t^2},\frac{1}{\beta_{1}},\frac{1}{\beta_{2}}},
    \end{eqnarray*}
    where $C=(\frac{t^{2}(\alpha_{1}+2t)(\alpha_{2}+2t)}{\alpha_{1}\alpha_{2}(1-t)^2})^{a-1}$ and $\beta_i=-\frac{\alpha_{i}(\alpha_{i}+2t)}{2\alpha_{i}+(1+t)^{2}}$.

    \item Theorem \ref{thm:transformula1(1,2)(1,0)(1,0)(0,1)}:
\begin{eqnarray*}
    &&\pFq{}{1}{\frac{3}{2}-a-b;1-a-b,1-b}{\frac{3}{2}-b}{s^{2},t^2}\\
    &=&(1-t)^{2b-2}(1+s)^{2a+2b-2}\cdot \pFq{}{1}{1-b;a,2-2a-2b}{2-2b}{-\frac{4t}{(1-t)^{2}},\frac{2(s-t)}{(1-t)(1+s)}},
\end{eqnarray*}

\item Theorem \ref{thm:sector region for (1,2)+(1,0)+(1,0)+(0,1)}:
\begin{eqnarray*}
    &&t_1^{2a+2b-1}(1-t_2^2)^{a+b-1}\pFq{}{1}{2a+2b-1;b,b}{a+2b}{\frac{(1-t_1)(t_2+1)}{2(t_2-t_1)},\frac{(1-t_1)(t_2-1)}{2(t_2+t_1)}}\\
    &=&(t_1^2-t_2^2)^{a+b-1}\pFq{}{1}{a+b;a+b-\frac{1}{2},1-a}{a+2b}{\frac{t_1^2-1}{t_1^2},\frac{t_1^2-1}{t_1^2-t_2^2}}.
\end{eqnarray*}
\end{enumerate}
Certain degenerations give rise to transformations for Euler--Gauss hypergeometric functions:
\begin{enumerate}
    \item Corollary \ref{Goursat1} (Goursat's quadratic transformation):
    \begin{eqnarray*}
    \pFq{2}{1}{a,b}{a+b-{1\over 2}}{z}=(1-z)^{-{1\over 2}} \pFq{2}{1}{2a-1, 2b-1}{a+b-{1\over 2}}{ {1\over 2}-{1\over 2} \sqrt{1-z}}.
\end{eqnarray*}

\item Corollary \ref{Goursat2} (Goursat's quadratic transformation):
\begin{eqnarray*}
    \pFq{2}{1}{a,b}{2b}{z}=(1-z)^{-\frac{a}{2}}\pFq{2}{1}{a,2b-a}{b+\frac{1}{2}}{-\frac{1}{4}(1-z)^{-\frac{1}{2}}(1-(1-z)^{\frac{1}{2}})^2}.
\end{eqnarray*}

\item Corollary \ref{not known yet}:
\begin{eqnarray*}
    \pFq{2}{1}{\frac{4}{3}-a,2-2a}{3-3a}{-\frac{t(2t+3)}{t+2}}
    = C\cdot \pFq{2}{1}{\frac{5}{6},\frac{3}{2}-a}{\frac{5}{3}-a}{-\frac{(2t+3)^{3}}{t^{3}(t+2)}},
\end{eqnarray*}
where $C=2^{2a-2}3^{\frac{5}{2}-3a}(-t)^{3a-\frac{9}{2}}(t+2)^{\frac{1}{2}-a}(t+3)^{2}$.
\end{enumerate}

The paper is organized as follows: In Section \ref{section: preliminary}, we review the definitions and integral representations of Gauss hypergeometric functions, Appell hypergeometric functions and Appell--Lauricella hypergeometric functions. In Section \ref{section: two lemmas}, we recall two lemmas to reduce single-variable integrals to Appell--Lauricella hypergeometric functions. In Section \ref{section: surfaces}, we outline the type of period integrals on surfaces. In the rest sections \ref{section:partition(1,2)(1,1)(1,0)}, \ref{section: (1,2)+(1,0)+(1,0)}, \ref{section: (2,2)+(1,0)+(0,1)}, \ref{section: (1,2)+(2,1)} and \ref{section:(1,3)+(1,0)+(1,0)}, we give several examples of transformation formulas obtained by the double integrals.

\textbf{Acknowledgements}: The second author is supported by the national key research and development program of China (No. 2022YFA1007100) and NSFC 12201337. The third author is partially supported by NSFC 12301058.

%In principle, every commensurability relation via the surfaces in \cite{yu2024comm} gives rise to a transformation formula. 

\section{A Review of Hypergeometric Functions}
\label{section: preliminary}

\subsection{Gauss Hypergeometric Function}
The Gauss hypergeometric function ${}_2F_1$ is defined by the series:
\[
\pFq{2}{1}{a, b}{c}{z} = \sum_{n=0}^{\infty} \frac{(a)_n (b)_n}{(c)_n} \frac{z^n}{n!},
\]
where $|z| < 1$, and $(q)_n = \frac{\Gamma(q+n)}{\Gamma(q)}$ is the Pochhammer symbol (rising factorial). The parameters $a, b, c \in \mathbb{C}$ with $c \neq 0, -1, -2, \dots$.

Another way of defining the Gauss hypergeometric function is the Euler integral representation:
\[
\pFq{2}{1}{a, b}{c}{z} = \frac{\Gamma(c)}{\Gamma(a)\Gamma(c-a)} \int_0^1 t^{a-1} (1-t)^{c-a-1} (1 - z t)^{-b}  dt,\ \ \ \mathrm{Re}(c) > \mathrm{Re}(a) > 0
\]
for $|z|<1$ or more generally, by analytic continuation, for $z$ in the cut plane $\mathbb{C}\setminus [1,\infty)$.

\subsection{Appell--Lauricella Hypergeometric Function}
The Appell hypergeometric function of two variables, denoted $F_1$, is defined by:
\[
\pFq{}{1}{a;b,b^{\prime}}{c}{x,y} = \sum_{m=0}^{\infty} \sum_{n=0}^{\infty} \frac{(a)_{m+n} (b)_m (b')_n}{(c)_{m+n}} \frac{x^m y^n}{m! \, n!},
\]
convergent for $|x| < 1$, $|y| < 1$, with $a, b, b', c \in \mathbb{C}$ and $c \neq 0, -1, -2, \dots$.

For $\mathrm{Re}(c) > \mathrm{Re}(a) > 0$, the Appell hypergeometric function $\pFq{}{1}{a;b,b^{\prime}}{c}{x,y}$ admits the following integral representation:
\[
\pFq{}{1}{a;b,b^{\prime}}{c}{x,y} = \frac{\Gamma(c)}{\Gamma(a)\Gamma(c-a)} \int_0^1 t^{a-1} (1-t)^{c-a-1} (1 - x t)^{-b} (1 - y t)^{-b'}  dt,
\]
which converges for $|x| < 1$, $|y| < 1$, or more generally, by analytic continuation, for $x, y \in \mathbb{C} \setminus [1, \infty)$.

This generalizes the Euler integral for ${}_2F_1$ and is valid when $\mathrm{Re}(c) > \mathrm{Re}(a) > 0$.

In general, the Appell--Lauricella hypergeometric function is defined as follows:
\begin{eqnarray*}
    \Fd{(n)}{D}{a;b_{1},\cdots,b_{n}}{c}{x_{1},\cdots,x_{n}}=\sum_{m_{1},\cdots,m_{n}=0}^{\infty} \frac{(a)_{m_{1}+\cdots+m_{n}} (b_{1})_{m_{1}}\cdots (b_{n})_{m_{n}}}{(c)_{m_{1}+\cdots +m_{n}}} \frac{x_{1}^{m_{1}}\cdots x^{m_{n}}_{n}}{m_{1}!\cdots \, m_{n}!}
\end{eqnarray*}

Again, for $\mathrm{Re}(c) > \mathrm{Re}(a) > 0$, the Appell--Lauricella hypergeometric function admits the following integral representation:
\begin{eqnarray*}
    &&\Fd{(n)}{D}{a;b_{1},\cdots,b_{n}}{c}{x_{1},\cdots,x_{n}}\\
    &=&\frac{\Gamma(c)}{\Gamma(a)\Gamma(c-a)} \int_0^1 t^{a-1} (1-t)^{c-a-1} (1 - x_1 t)^{-b_1}(1-x_2t)^{-b_2}\cdots (1 - x_n t)^{-b_n}  dt.
\end{eqnarray*}
This integral representation also defines an analytic continuation of Appell--Lauricella hypergeometric functions for $x_i\in \CC\bs [1,\infty)$.

\section{Two Lemmas}
\label{section: two lemmas}
The results in this section relate certain period integrals of single variable to Appell--Lauricella hypergeometric functions.
\begin{lem}
\label{lemma: 4 poles}
Suppose $n\geq4$. Let $x_1,x_{2},\cdots,x_n$ be distinct real numbers that satisfy $x_{1}<x_{2}$ and $x_{3},\cdots,x_{n}\notin (x_{1},x_{2})$. Let $\mu_1,\cdots,\mu_n$ be real numbers that satisfy $\mu_{1},\mu_{2}<1$ and the sum of all $\mu_i$ is equal to 2. Let $$R_i=\frac{(x_i-x_3)(x_2-x_1)}{(x_i-x_1)(x_2-x_3)}.$$ 

Let $$\omega= {dx\over (x-x_1)^{\mu_1} (x_2-x)^{\mu_2} |x-x_3|^{\mu_3}|x-x_4|^{\mu_4}\cdots|x-x_n|^{\mu_n}}.$$

Then 
\begin{eqnarray*}
    \int_{x_1}^{x_2}\omega=C\cdot \Fd{(n-3)}{D}{1-\mu_1;\mu_4,\cdots,\mu_n}{2-\mu_1-\mu_2}{R_4,\cdots,R_n},
\end{eqnarray*}
where 
\begin{eqnarray*}
    C&=&|x_2-x_3|^{\mu_1-1}(x_2-x_1)^{1-\mu_1-\mu_2}|x_1-x_3|^{1-\mu_1-\mu_3}|x_1-x_4|^{-\mu_4}
    \\&&|x_1-x_5|^{-\mu_5}\cdots|x_1-x_n|^{-\mu_n}B(1-\mu_1,1-\mu_2).
\end{eqnarray*}
\end{lem}

\begin{proof}
Consider M\"obius transformation $x\mapsto y={x-x_1\over x-x_3} {x_2-x_3\over x_2-x_1}$, then $x_1, x_2, x_3, x_4,x_5$ $x_6,\cdots, x_n$ are mapped to $0,1,\infty, R_4^{-1},R_{5}^{-1},R_{6}^{-1},\cdots,R_n^{-1}$ respectively.

Then $x=f(y)=x_3+{\alpha(x_1-x_3)\over \alpha-y}$ with $\alpha={x_2-x_3\over x_2-x_1}$. One has $dx=f'(y)dy= {|\alpha| |x_1-x_3|\over (\alpha-y)^2} dy$.

With respect to $x=f(y)$, the differential form $\omega$ is transformed to 
\begin{equation*}
|x_2-x_3|^{\mu_1-1}(x_2-x_1)^{1-\mu_1-\mu_2}|x_1-x_4|^{1-\mu_1-\mu_3}\prod\limits_{i=4}^{n}|x_1-x_i|^{-\mu_n}{dy\over y^{\mu_1} (1-y)^{\mu_2}\prod\limits_{i=3}^{n}(1-R_{i}y)^{\mu_i}}
\end{equation*}
This proves the lemma.
\end{proof}

When we let $x_3\to \infty$, we obtain formula when one of the poles is at infinity.

\begin{cor}
\label{corollary: poles + infty}

Suppose $n\geq3$. Let $x_1,x_2,x_4\cdots,x_{n}$ be distinct real numbers that satisfy $x_{1}<x_{2}$ and $x_{4},x_{5},\cdots,x_{n}\notin (x_{1},x_{2})$. Assume $\mu_1,\mu_2,\mu_3,\cdots,\mu_{n}$ are real numbers such that $\mu_{1},\mu_{2}<1$ and sum of all $\mu_i$ is equal to $2$. Denote $$R_i=\frac{x_2-x_1}{x_i-x_1}$$ for $4\leq i\leq n$.

Let $$\omega= {dx\over (x-x_1)^{\mu_1} (x_2-x)^{\mu_2} |x-x_4|^{\mu_4}\cdots|x-x_{n}|^{\mu_{n}}},$$then 
\begin{eqnarray*}
    \int_{x_1}^{x_2}\omega=C\cdot \Fd{(n-3)}{D}{1-\mu_1;\mu_4,\cdots,\mu_n}{2-\mu_1-\mu_2}{R_4,\cdots,R_n},
\end{eqnarray*}
where 
\begin{eqnarray*}
    C&=&(x_2-x_1)^{1-\mu_1-\mu_3}|x_1-x_4|^{-\mu_4}
    |x_1-x_5|^{-\mu_5}\cdots|x_1-x_{n}|^{-\mu_{n}}B(1-\mu_1,1-\mu_2).
\end{eqnarray*}
\end{cor}

The following lemma serves as a function analog of the Leray--Hirsch theorem used in \cite[Theorem 3.5]{yu2024comm}. The proof is the same as Lemma \ref{lemma: 4 poles} by M\"obius transformation $y={x-x_1\over x-x_3} {x_2-x_3\over x_2-x_1}$.
\begin{lem}
\label{lemma: 3 poles}
Suppose $x_1,x_2,x_3$ are three distinct real numbers that satisfy $x_{1}<x_{2}$ and $x_{3}\notin (x_{1},x_{2})$. Let $\mu_1,\mu_2,\mu_3$ be real numbers such that $\mu_{1},\mu_{2}<1$ and $\mu_1+\mu_2+\mu_3=2$.

Let $$\omega= {dx\over (x-x_1)^{\mu_1} (x_2-x)^{\mu_2} |x-x_3|^{\mu_3}},$$ then
\begin{eqnarray*}
 \int_{x_1}^{x_2} \omega 
&=& (x_2-x_1)^{\mu_3-1} |x_2-x_3|^{\mu_1-1} |x_1-x_3|^{\mu_2-1} \int_0^1 {dy\over y^{\mu_1} (1-y)^{\mu_2}} \\
&=&  (x_2-x_1)^{\mu_3-1} |x_2-x_3|^{\mu_1-1} |x_1-x_3|^{\mu_2-1}B(1-\mu_1,1-\mu_2).
\end{eqnarray*}
\end{lem}

\section{Double-fibration and period integrals}
\label{section: surfaces}
Let us consider divisors $D_i$ ($1\leq i\leq m$) on $\mathbb{P}^1\times \mathbb{P}^1$ and rational number $\mu_i\in (0,1)$ satisfying the following conditions (see \cite[Proposition~3.1]{yu2024comm}):
\begin{eqnarray}
\label{numerical conditions on divisors}
    \sum_{i=1}^m \deg D_i =(3,3), \,\,
    \sum_{i=1}^m \mu_i\deg D_i = (1,1)
\end{eqnarray}
Let $d$ be the common denominator of $\mu_i$ ($1\le i\le m$), and let $S$ be the $d$-fold cyclic cover of $\PP^1\times \PP^1$ branching along each $D_i$ with multiplicity $d\mu_i$. The cyclic group $\ZZ/d\ZZ$ operates on $S$. Conditions \eqref{numerical conditions on divisors} imply that $\dim H^{2,0}_\chi(S)=1$ for tautological character $\chi$ of $\ZZ/d\ZZ$. %In this sense, we say that the Hodge structure on $H^2_\chi(S)$ is of Calabi-Yau type. 

Assume that the divisors $D_i$ are defined by equation $f_i(x,y)=0$, where $(x,y)$ is the affine coordinate of $\mathbb{P}^1\times \mathbb{P}^1$. Then an affine open subset of $S$ is given by 
\[
z^d=\prod_{i=1}^m f_i(x,y)^{d\mu_i}.
\]
A generator of $H^{2,0}_\chi(S)$ is given by 
\begin{equation*}
    \omega={dx\wedge dy\over \prod_{i=1}^m f_i(x,y)^{1-\mu_i}}.
\end{equation*}
A slight generalization of cyclic cover to allow real parameters $\mu_i$ is as follows. Denote $U=(\mathbb{P}^1\times \mathbb{P}^1) \setminus \bigcup_{i=1}^m D_i$. Let $\mathbb{L}$ be the rank-one local system on $U$ with local monodromy around $D_i$ given by $e^{2\pi \sqrt{-1} \mu_i}$. Since the surface $U$ is smooth and affine, the cohomology $H^2(U,\mathbb{L})$ admits a de-Rham description via de-Rham complex twisted by $\mathbb{L}$. In particular, the two-form $\omega$ defines a cohomology class in $H^2(U,\mathbb{L})$. Let $\Delta$ be a two-simplex and consider a continuous map 
\[
\sigma\colon \Delta \to \mathbb{P}^1\times \mathbb{P}^1
\]
such that $\sigma(\partial \Delta)\subset \bigcup_{i=1}^m D_i$ and the interior of $\Delta$ is contained in $U$. Choose a section $e$ of $\sigma^*\mathbb{L}$ on the interior of $\Delta$. Then the cycle $\Delta\otimes e$ represents a section of $H_2^{\text{BM}}(U, \LL)$. When $D=\sum\limits_{i=1}^m D_i$ is simple normal crossing, this cycle can be lifted to $H_2(U, \LL)$. More generally, we recall the following well-known fact (see, for example, \cite[Lemma 2.4]{xie2025localsystem} for a proof):

\begin{lem}
    Let $X$ be a projective smooth surface and $D=\sum\limits_{i=1}^m D_i$ be a simple normal crossing divisor on $X$ with smooth components $D_i$. Denote by $U$ the complement of $D$ and $j\colon U\xrightarrow{} X$ the open inclusion. Assume that $\LL$ is a rank-one complex local system on $U$ and the local monodromy around $D_i$ is given by $e^{2\pi \sqrt{-1} \mu_i}$ with real numbers $0<\mu_i<1$. Then the natural map $H_n(U, \LL)\to H_n^{BM}(U, \LL)$ is an isomorphism.
\end{lem}

When $D$ is not simple normal crossing, we consider the blowup of $\PP^1\times \PP^1$ such that the total transform of $D$ is simple normal crossing. Then we can lift the cycle $\Delta\otimes e$ to the Borel--Moore homology of the complement of the total transform with coefficients in the pullback local system. If the monodromy of $\LL$ around each exceptional divisor is nontrivial, then the above lemma still guarantees that the cycle $\Delta\otimes e$ can be further lifted to the homology group.

The pairing of $\omega$ with $\Delta\otimes e$ gives rise to the period integral
\begin{equation*}
   I = \int_{\Delta} \omega=\int_{\Delta} { dx\wedge dy\over \prod_{i=1}^m f_i(x,y)^{1-\mu_i}}. 
\end{equation*}

 By \cite[Theorem 3.6]{yu2024comm}, iterated integration along $x$ (or $y$) direction reduce the period integral to period integral on punctured $\PP^1$ with form

\begin{equation*}
    \int_{\gamma} {C dz\over \prod_{j=1}^n (z-z_j)^{\nu_j}}.
\end{equation*}

Here the points $z_j\in \PP^1$ are discriminant points of the map $D\to \PP^1$ induced by the projection on $y$ (or $x$) coordinate. The real numbers $\nu_j\in (0,1)$ are certain integral combinations of $\mu_i$ and ${1\over 2}$, see \cite[Section 4]{yu2024comm} for the detailed relation between $\mu_i$ and $\nu_j$. The constant $C$ on $\PP^1$ depends on the choice of $f_i$. 

Let the sections $f_i$ vary and keep the singularity type of $D$. Denote by $\calM$ the GIT moduli space of such $f_i$. Assume on suitable affine open subset $\mathbb{A}$ of $\calM$, we have a flat choice of $\Delta$, and algebraic functions $f_i$. The period integral $I$ becomes a multivalued function on $\mathbb{A}$. By \cite[Proposition 5.3]{yu2024comm}, the map from $\calM$ to moduli of points $(z_j)\subset \PP^1$ is a finite algebraic map. The constant depending on $f_i$ turns out to be algebraic functions on the moduli. So after suitable change of variables, the above double integral can be expressed in terms of Appell--Lauricella hypergeometric functions. Then the two different iterated integral orders give a transformation formula for two Appell--Lauricella hypergeometric functions. 

Following this approach, we can obtain transformation formulas for all tuples of divisors listed in \cite[Proposition~3.3]{yu2024comm}. The choice of different $\Delta$ could possibly give different transformation formulas for the same tuple of divisors. On the other hand, the monodromy representations arising from Deligne--Mostow theory are irreducible, so all the transformation formulas arising from different choices of $\Delta$ are expected to be equivalent under monodromy action.

In the remaining sections, we give some examples of transformation laws of hypergeometric functions arising from double integrals. We divide the examples according to the partitions of $(3,3)$ (see \cite[Proposition~3.2]{yu2024comm}).

\section{Partition $(1,2)+(1,1)+(1,0)$}
\label{section:partition(1,2)(1,1)(1,0)}
In this section, we consider the case where the divisor $D$ decomposes into three components $D_{1},D_{2},D_{3}$ with degrees $(1,2)$, $(1,1)$ and $(1,0)$ respectively.

\subsection{Parametrization and cycle: Figure \ref{figure: (2,1)(1,1)(0,1)segment}}
Let $t,\alpha_{1},\alpha_{2}$ be real numbers satisfying $t<-1$ and $\alpha_{1}<\alpha_{2}<-\frac{(1+t)^{2}}{2}$. Denote by $$\beta_{i}=-\frac{\alpha_{i}(\alpha_{i}+2t)}{2\alpha_{i}+(1+t)^{2}}.$$
We define three divisors under affine coordinates $(x,y)\in \CC^2$ as follows:
\begin{eqnarray*}
    D_{1} &:& y^{2}+2(x+t)y+(1+t)^{2}x=0\\
    D_{2} &:& q_{1}y-q_{2}x-q_{3}=0\\
    D_{3} &:& x=0,
\end{eqnarray*}
where $q_{1},q_{2},q_{3}$ equal $\beta_{2}-\beta_{1}$, $\alpha_{2}-\alpha_{1}$ and $\alpha_{1}\beta_{2}-\alpha_{2}\beta_{1}$ respectively.

Let $a<1$ be a real number. We consider the period integral of the two-form:
\begin{eqnarray*}
    \omega=\frac{dx\wedge dy}{x^{a}(-y^{2}-2(x+t)y-(1+t)^{2}x)^{a}(q_{1}y-q_{2}x-q_{3})^{2-2a}}.
\end{eqnarray*}

Fix $x\in(0,1)$, let $0<y_{-}(x)<y_{+}(x)$ be two real roots of $y^{2}+2(x+t)y+(1+t)^{2}x=0$. The divisor $D_{1}$ can be written as $x=x(y)=-\frac{y^{2}+2ty}{2y+(1+t)^{2}}$. The specified integration cycle $\sigma$ (see Figure \ref{figure: (2,1)(1,1)(0,1)segment}) is defined by
\begin{eqnarray*}
    \sigma &=& \{(x,y)\in \RR^{2}\ | \ 0<x<1\ , \ y_{-}(x)<y<y_{+}(x) \}\\
           &=& \{(x,y)\in \RR^{2}\ | \ 0<y<-2t\ , \ 0<x<x(y)\}.
\end{eqnarray*}

\begin{figure}[H] %H为当前位置, !htb为忽略美学标准, htbp为浮动图形
\centering %图片居中
\includegraphics[width=0.4\textwidth]{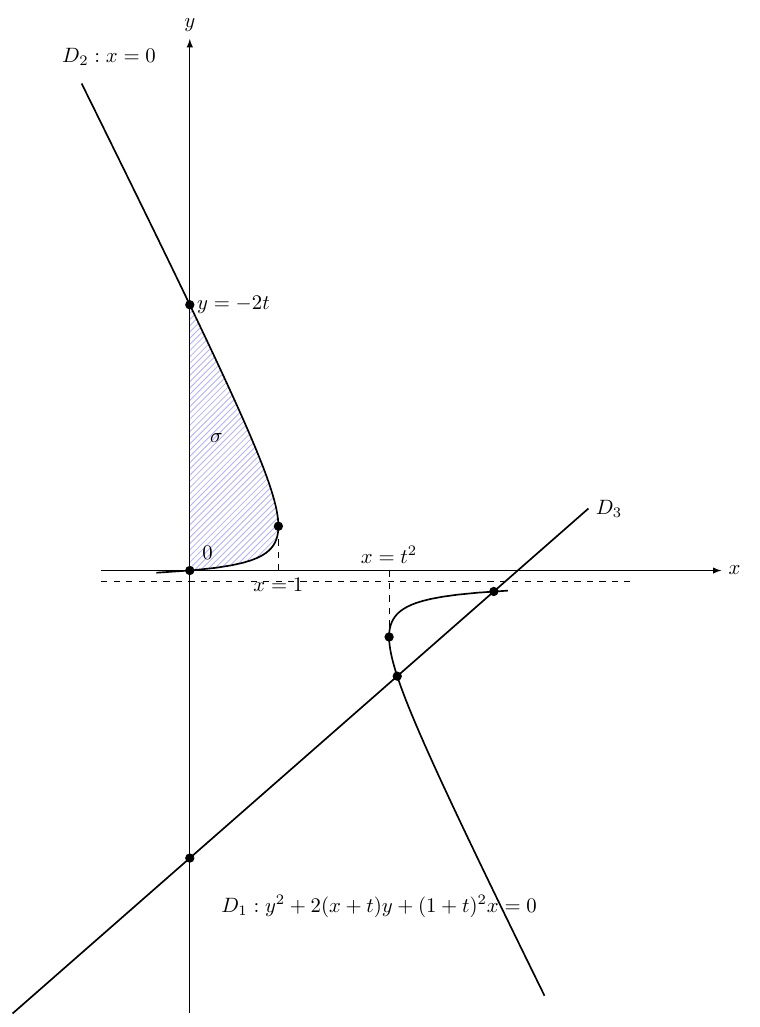} %插入图片, []中设置图片大小, {}中是图片文件名
\caption{$(2,1)+(1,1)+(0,1)$} %最终文档中希望显示的图片标题
\label{figure: (2,1)(1,1)(0,1)segment}
\end{figure}

\subsection{Period integral in Figure \ref{figure: (2,1)(1,1)(0,1)segment}}
\subsubsection{First Integration Order}
We first integrate with respect to $y$.
\begin{eqnarray*}
    \int_{\sigma}\omega=\int_{0}^{1}\frac{dx}{q_{1}^{2-2a}x^{a}}\int_{y_{-}(x)}^{y_{+}(x)}\frac{dy}{(y-y_{-}(x))^{a}(y_{+}(x)-y)^{a}(y-\frac{q_{2}x+q_{3}}{q_{1}})^{2-2a}}.
\end{eqnarray*}
Using Lemma \ref{lemma: 3 poles} (with exponents $a,a,2-2a$ summing to $2$), the inner integral evaluates to:
\begin{eqnarray*}
    I_{y}(x)=\frac{B(1-a,1-a)}{x^{a}(y_{+}(x)-y_{-}(x))^{2a-1}((q_{1}y_{-}(x)-(q_{2}x+q_{3})(q_{1}y_{+}(x)-(q_{2}x+q_{3})))^{1-a}}.
\end{eqnarray*}
Substituting $y_{+}(x)+y_{-}(x)=-2(x+t)$, $y_{+}(x)y_{-}(x)=(1+t)^{2}x$ and $y_{+}(x)-y_{-}(x)=2\sqrt{(1-x)(t^2-x)}$, we have:
\begin{eqnarray*}
    I_{y}(x)=\frac{2^{1-2a}B(1-a,1-a)}{(q_{2}^{2}+2q_{1}q_{2})^{1-a}x^{a}((1-x)(t^{2}-x))^{a-\frac{1}{2}}((\beta_{1}-x)(\beta_{2}-x))^{1-a}}.
\end{eqnarray*}
Substituting $I_{y}(x)$ back into the integral, we obtain:
\begin{equation*}
\int_{\sigma}\omega=\frac{2^{1-2a}B(1-a,1-a)}{(q_{2}^{2}+2q_{1}q_{2})^{1-a}}\int_{0}^{1}\frac{dx}{x^{a}(1-x)^{a-\frac{1}{2}}(t^{2}-x)^{a-\frac{1}{2}}((\beta_{1}-x)(\beta_{2}-x))^{1-a}}.
\end{equation*}
This integral has singularities at $0,1,\infty,t^{2},\beta_{1},\beta_{2}$. By Corollary \ref{corollary: poles + infty}, we can express this in terms of the Appell--Lauricella hypergeometric function $F_{D}^{(3)}$:
\begin{eqnarray*}
    \int_{\sigma}\omega=C_{1}\Fd{(3)}{D}{1-a;a-\frac{1}{2},1-a,1-a}{\frac{5}{2}-2a}{\frac{1}{t^2},\frac{1}{\beta_{1}},\frac{1}{\beta_{2}}}.
\end{eqnarray*}
The constant $C_{1}$ is given by
\begin{eqnarray*}
    C_{1}=B(1-a,1-a)B(1-a,\frac{3}{2}-a)(2t)^{1-2a}((q_{2}^{2}+2q_{1}q_{2})\beta_{1}\beta_{2})^{a-1}.
\end{eqnarray*}

\subsubsection{Second Integration Order}
Integrating with respect to $x$ first, we have
\begin{eqnarray*}
    \int_{\sigma}\omega=\int_{0}^{-2t}\frac{dy}{q_{2}^{2-2a}(2y+(1+t)^{2})^{a}}\int_{0}^{x(y)}\frac{dx}{x^{a}(x(y)-x)^{a}(\frac{q_{1}y-q_{3}}{q_{2}}-x)^{2-2a}}.
\end{eqnarray*}
Using Lemma \ref{lemma: 3 poles} (with exponents $a$, $a$, $2-2a$ summing to $2$), the inner integral evaluates to:
\begin{eqnarray*}
    I_{x}(y)=B(1-a,1-a)((1+t)^{2}+2y)^{a}(x(y))^{2a-1}(q_{1}y-q_{3})^{1-a}(q_{1}y-q_{3}-q_{2}x(y))^{1-a}.
\end{eqnarray*}
Substitute $I_{x}(y)$ back into the integral and note that \[x(y)=-\frac{y^{2}+2ty}{2y+(1+t)^{2}}.\] Simplify the integral and we obtain:
\begin{eqnarray*}
    \int_{\sigma}\omega=\frac{B(1-a,1-a)}{q_{1}^{1-a}(2q_{1}+q_{2})^{1-a}}\int_{0}^{-2t}\frac{dy}{y^{2a-1}(-2t-y)^{2a-1}(y-\frac{q_{3}}{q_{1}})^{1-a}((y-\alpha_{1})(y-\alpha_{2}))^{1-a}}.
\end{eqnarray*}
This integral has singularities $0,-2t,\infty,\frac{q_{3}}{q_{1}},\alpha_{1},\alpha_{2}$. By Corollary \ref{corollary: poles + infty}, we can express this in terms of the Appell--Lauricella hypergeometric function $F_{D}^{(3)}$:
\begin{eqnarray*}
    \int_{\sigma}\omega=C_{2}\Fd{(3)}{D}{2-2a;1-a,1-a,1-a}{4-4a}{-\frac{2q_{1}t}{q_{3}},-\frac{2t}{\alpha_{1}},-\frac{2t}{\alpha_{2}}}.
\end{eqnarray*}
The constant $C_{2}$ is given by
\begin{eqnarray*}
    C_{2}=\frac{B(1-a,1-a)B(2-2a,2-2a)(-2t)^{3-4a}}{(-q_{3}(2q_{1}+q_{2})\alpha_{1}\alpha_{2})^{1-a}}.
\end{eqnarray*}

\subsubsection{Transformation Formula}
Comparing two results of $\int_{\sigma}\omega$, we obtain the following transformation formula:
\begin{thm}
\label{thm:transformula(1,2)(1,1)(1,0)}
    Let $t,\alpha_{1},\alpha_{2},a$ be real numbers such that $t<-1$, $\alpha_{1}<\alpha_{2}<-\frac{(1+t)^{2}}{2}$ and $a<1$. Let $\beta_{i}$ be $-\frac{\alpha_{i}(\alpha_{i}+2t)}{2\alpha_{i}+(1+t)^{2}}$, $i=1,2$. We have the following transformation formula relating $F_{D}^{(3)}$:
    \begin{eqnarray*}
        &&\Fd{(3)}{D}{2-2a;1-a,1-a,1-a}{4-4a}{\frac{2(\beta_{2}-\beta_{1})t}{\alpha_{2}\beta_{1}-\alpha_{1}\beta_{2}},-\frac{2t}{\alpha_{1}},-\frac{2t}{\alpha_{2}}}\\
        &=&C\cdot \Fd{(3)}{D}{1-a;a-\frac{1}{2},1-a,1-a}{\frac{5}{2}-2a}{\frac{1}{t^2},\frac{1}{\beta_{1}},\frac{1}{\beta_{2}}},
    \end{eqnarray*}
    where $C=\frac{C_{1}}{C_{2}}=(\frac{t^{2}(\alpha_{1}+2t)(\alpha_{2}+2t)}{\alpha_{1}\alpha_{2}(1-t)^2})^{a-1}$.
\end{thm}

\section{Partition $(1,2)+(1,0)+(1,0)+(0,1)$}
\label{section: (1,2)+(1,0)+(1,0)}

In this section, we give an example of transformation formulas of hypergeometric functions arising from double integrals corresponding to divisors $D_1$, $D_2$, $D_3$, $D_{4}$ defined with degrees $(1,2)$, $(1,0)$, $(0,1)$, $(1,0)$ respectively.

\subsection{Period integral in Figure \ref{figure: (1,2)(1,0)(1,0)(0,1)segment}}
%\subsection{Parametrization and cycle: Figure \ref{figure: (1,2)(1,0)(1,0)(0,1)segment}}

Let $a$, $b$ be two real numbers satisfying $a$, $b>0$ and $\frac{1}{2}<a+b<1$. Consider two parameters $t_1$ and $t_2$ satisfying $0<t_1<1$ and $0<t_2^2<t_1^2$. Let the divisors be given by
\begin{eqnarray*}
    D_1 &:& x - y^2 = 0 \\
    D_2 &:& x - t_2^2 = 0 \\
    D_3 &:& y - t_1 = 0 \\
    D_4 &:& x - 1 = 0
\end{eqnarray*}

Let $\omega$ be a two-form
\[
    \omega=\frac{dx\wedge dy}{(x-y^2)^{a+b}(1-x)^{1-b}(t_{2}^{2}-x)^{1-a}(y-t_1)^{2-2a-2b}}.
\]

\begin{figure}[H] %H为当前位置, !htb为忽略美学标准, htbp为浮动图形
\centering %图片居中
\includegraphics[width=0.25\textwidth]{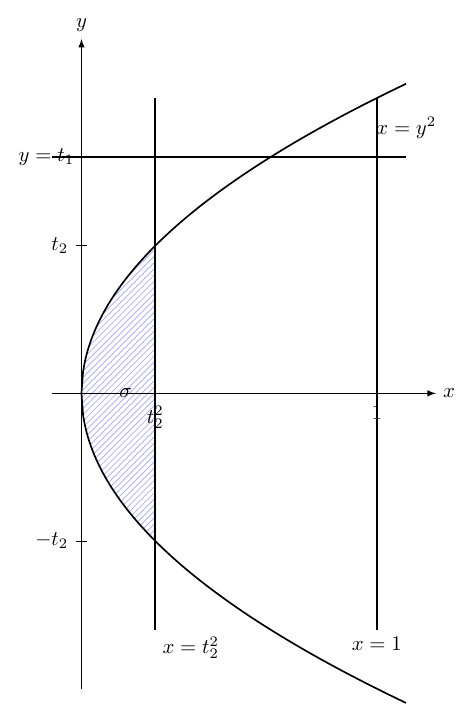} %插入图片, []中设置图片大小, {}中是图片文件名
\caption{$(1,2)+(1,0)+(1,0)+(0,1)$, Segment Region} %最终文档中希望显示的图片标题
\label{figure: (1,2)(1,0)(1,0)(0,1)segment}
\end{figure}

Let $\sigma$ be a region (see Figure \ref{figure: (1,2)(1,0)(1,0)(0,1)segment}) in $\RR^{2}$ defined by
\begin{eqnarray*}
    \sigma &=&\{(x,y)\in \RR^2|\  0<x<t_{2}^2,-\sqrt{x}<y<\sqrt{x}\}\\
    &=&\{(x,y)\in \RR^2|\ -t_{2}<y<t_{2}, y^2<x<t_{2}^{2} \}
\end{eqnarray*}

\subsubsection{First Integration Order}
First integrate with respect to $y$.
\begin{eqnarray*}
    \int_{\sigma}\omega=\int_{0}^{t_{2}^{2}}\frac{dx}{(1-x)^{1-b}(t_{2}^{2}-x)^{1-a}}\int_{-\sqrt{x}}^{\sqrt{x}}\frac{dy}{(y+\sqrt{x})^{a+b}(\sqrt{x}-y)^{a+b}(t_{1}-y)^{2-2a-2b}}.
\end{eqnarray*}
Use Lemma \ref{lemma: 3 poles} for the inner integral, the integral becomes:
\begin{eqnarray*}
    2^{1-2a-2b}B(1-a-b,1-a-b)\int_{0}^{t_{2}^{2}}\frac{dx}{x^{a+b-\frac{1}{2}}(t_{2}^{2}-x)^{1-a}(t_{1}^{2}-x)^{1-a-b}(1-x)^{1-b}}.
\end{eqnarray*}
This integral has singularities at $0,t_{2}^{2},\infty,t_{1}^{2},1$. By Lemma \ref{corollary: poles + infty}, we obtain the following representation of the Appell function $F_{1}$:
\begin{eqnarray*}
    C_{1}\cdot \pFq{}{1}{\frac{3}{2}-a-b;1-a-b,1-b}{\frac{3}{2}-b}{(\frac{t_{2}}{t_{1}})^{2},t_{2}^{2}},
\end{eqnarray*}
where 
\begin{eqnarray*}
    C_{1}=2^{1-2a-2b}B(1-a-b,1-a-b)B(\frac{3}{2}-a-b,a)(t_{2})^{1-2b}(t_{1})^{2a+2b-2}.
\end{eqnarray*}

\subsubsection{Second Integration Order}
Integrate with respect to $x$.
\begin{eqnarray*}
    \int_{\sigma}\omega=\int_{-t_{2}}^{t_{2}}\frac{dy}{(t_{1}-y)^{2-2a-2b}}\int_{y^{2}}^{t_{2}^{2}}\frac{dx}{(x-y^{2})^{a+b}(t_{2}^{2}-x)^{1-a}(1-x)^{1-b}}.
\end{eqnarray*}
Using Lemma \ref{lemma: 3 poles} for the inner integral, the integral becomes:
\begin{eqnarray*}
    B(1-a-b,a)(1-t_{2}^{2})^{a+b-1}\int_{-t_{2}}^{t_{2}}\frac{dy}{(y+t_{2})^{b}(t_{2}-y)^{b}((1-y)(1+y))^{a}(t_{1}-y)^{2-2a-2b}}.
\end{eqnarray*}
This integral has singularities at $-t_{2},t_{2},1,-1,t_{1}$. By Lemma \ref{lemma: 4 poles}, we obtain the following representation of the Appell function $F_{1}$:
\begin{eqnarray*}
    C_{2}\cdot \pFq{}{1}{1-b;a,2-2a-2b}{2-2b}{-\frac{4t_{2}}{(1-t_{2})^{2}},\frac{2t_{2}(1-t_{1})}{(1-t_{2})(t_{1}+t_{2})}},
\end{eqnarray*}
where
\begin{eqnarray*}
    C_{2}=B(1-a-b,a)B(1-b,1-b)(2t_{2})^{1-2b}(1-t_{2})^{2b-2}(t_{1}+t_{2})^{2a+2b-1}.
\end{eqnarray*}

\subsubsection{Transformation Formula}
Setting $s=\frac{t_{2}}{t_{1}}$ and $t=t_{2}$, we obtain the following transformation formula relating Appell hypergeometric functions.

\begin{thm}
\label{thm:transformula1(1,2)(1,0)(1,0)(0,1)}
For real numbers $a$, $b$, $t_1$, $t_2$ in the following range $a>0, a+b<1, 0<t<1$ and $t<s<1$, we have
\begin{eqnarray*}
    &&\pFq{}{1}{\frac{3}{2}-a-b;1-a-b,1-b}{\frac{3}{2}-b}{s^{2},t^2}\\
    &=&C\cdot \pFq{}{1}{1-b;a,2-2a-2b}{2-2b}{-\frac{4t}{(1-t)^{2}},\frac{2(s-t)}{(1-t)(1+s)}},
\end{eqnarray*}
where $C=(1-t)^{2b-2}(1+s)^{2a+2b-2}$.
\end{thm}

\begin{rmk}
    The transformation formula in Theorem \ref{thm:transformula(1,2)(1,1)(1,0)} can degenerate to a special case of the formula in Theorem \ref{thm:transformula1(1,2)(1,0)(1,0)(0,1)}.

    If $\alpha_{1}<-t(t+1)<\alpha_{2}$ and let $\beta_{1}\rightarrow \beta_{2}$, then divisor $D_{2}$ degenerates to $x=\beta_{1}$ union $y=\infty$. This happens in the compactification $\PP^{1}\times \PP^{1}$. Actually, in the projective coordinate, $D_{2}$ is defined by $q_{1}yz_{1}-q_{2}xz_{2}-q_{3}z_{1}z_{2}=0$. It degenerates to $(x-\beta_{1})z_{2}=0$ as $\beta_{1}\rightarrow \beta_{2}$.

    After a coordinate transformation, the form
    \[\omega=\frac{dx\wedge dy}{x^{a}(-y^{2}-2(x+t)y-(1+t)^{2}x)^{a}(q_{1}y-q_{2}x-q_{3})^{2-2a}}\] 
    degenerates to 
    \[\omega=\frac{dx\wedge dy}{(x-y^2)^{a}(1-x)^{2-2a}(t_{2}^{2}-x)^{a}(y-t_1)^{2-2a}}.\]
    This form gives the special case of $2a+b=1$ in Theorem \ref{thm:transformula1(1,2)(1,0)(1,0)(0,1)}.
\end{rmk}

\subsection{Period integral in Figure \ref{figure: (1,2)(1,0)(1,0)(0,1)sector}}

Let $\sigma$ be another region (see Figure \ref{figure: (1,2)(1,0)(1,0)(0,1)sector}) in $\RR^2$ defined by 
\begin{eqnarray*}
    \sigma &=&\{(x,y)\in \RR^2|\  t_1^2<x<1,t_1<y<\sqrt{x}\}\\
    &=&\{(x,y)\in \RR^2|\ t_1<y<1, y^2<x<1 \}
\end{eqnarray*}
and $\omega$ be a two-form
\[
    \omega=\frac{dx\wedge dy}{(x-y^2)^{a+b}(1-x)^{1-b}(x-t_2^{2})^{1-a}(y-t_1)^{2-2a-2b}}.
\]
\begin{figure}[H] %H为当前位置, !htb为忽略美学标准, htbp为浮动图形
\centering %图片居中
\includegraphics[width=0.3\textwidth]{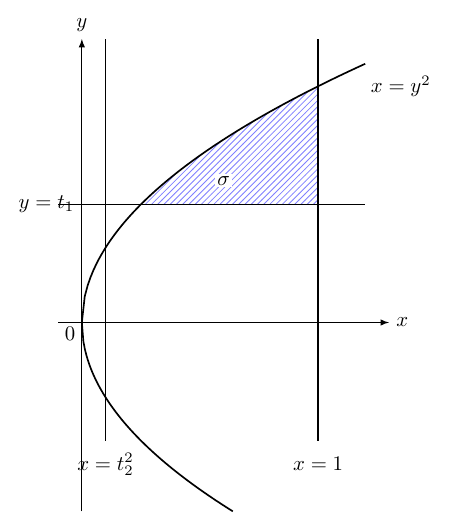} %插入图片, []中设置图片大小, {}中是图片文件名
\caption{$(1,2)+(1,0)+(1,0)+(0,1)$, Sector Region} %最终文档中希望显示的图片标题
\label{figure: (1,2)(1,0)(1,0)(0,1)sector}
\end{figure}
Next we give the transformation relation based on Figure \ref{figure: (1,2)(1,0)(1,0)(0,1)sector}.

Integrating with respect to $y$.
\begin{eqnarray*}
    \int_\sigma \omega&=&\int_{t_1^2}^{1}\frac{dx}{(1-x)^{1-b}(x-t_2^2)^{1-a}} \int_{t_1}^{\sqrt{x}} \frac{dy}{(y-t_1)^{2-2a-2b}(x-y^2)^{a+b}}
\end{eqnarray*}
Using Lemma \ref{lemma: 3 poles} for the inner integral, the integral becomes:
\begin{eqnarray*}
2^{1-2a-2b}B(2a+2b-1,1-a-b)\int_{t_{1}^{2}}^{1}\frac{dx}{(x-t_{1}^{2})^{1-a-b}(1-x)^{1-b}x^{a+b-\frac{1}{2}}(x-t_{2}^{2})^{1-a}}.
\end{eqnarray*}
This integral has singularities at $t_{1}^{2},1,\infty,0,t_{2}^{2}$. By Corollary \ref{corollary: poles + infty}, we obtain the following representation of the Appell function $F_{1}$:
\begin{eqnarray*}
    C_{1}\cdot \pFq{}{1}{a+b;a+b-\frac{1}{2},1-a}{a+2b}{\frac{t_{1}^{2}-1}{t_{1}^{2}},\frac{t_{1}^{2}-1}{t_{1}^{2}-t_{2}^{2}}},
\end{eqnarray*}
where
\begin{eqnarray*}
    C_{1}=2^{1-2a-2b}B(2a+2b-1,1-a-b)B(a+b,b)t_{1}^{1-2a-2b}(1-t_{1}^{2})^{a+2b-1}(t_{1}^{2}-t_{2}^{2})^{a-1}.
\end{eqnarray*}

Consider the second integration order. First integrate with respect to $x$.
\begin{eqnarray*}
    \int_\sigma \omega&=&\int_{t_1}^{1}\frac{dy}{(y-t_1)^{2-2a-2b}} \int_{y^2}^{1} \frac{dx}{(x-y^2)^{a+b}(1-x)^{1-b}(x-t_2^2)^{1-a}}.
\end{eqnarray*}
Using Lemma \ref{lemma: 3 poles} for the inner integral, the integral becomes:
\begin{eqnarray*}
\frac{B(1-a-b,b)}{(1-t_{2}^{2})^{1-a-b}}\int_{t_{1}}^{1}\frac{dy}{(y-t_{1})^{2-2a-2b}(1-y)^{a}(1+y)^{a}(y^{2}-t_{2}^{2})^{b}}.
\end{eqnarray*}
This integral has singularities at $t_{1},1,-1,t_{2},-t_{2}$. By Lemma \ref{lemma: 4 poles}, we obtain the following representation of the Appell function $F_{1}$:
\begin{eqnarray*}
    C_{2}\cdot \pFq{}{1}{2a+2b-1;b,b}{a+2b}{\frac{(1-t_{1})(1+t_{2})}{2(t_{2}-t_{1})},-\frac{(1-t_{1})(1-t_{2})}{2(t_{2}+t_{1})}},
\end{eqnarray*}
where
\begin{eqnarray*}
    C_{2}=2^{1-2a-2b}B(1-a-b,b)B(2a+2b-1,1-a)(1-t_{1}^{2})^{a+2b-1}(t_{1}^{2}-t_{2}^{2})^{-b}.
\end{eqnarray*}

Comparing the two representations of $\int_{\sigma}\omega$, we obtain the following transformation formula of the Appell hypergeometric functions:
\begin{thm}
\label{thm:sector region for (1,2)+(1,0)+(1,0)+(0,1)}
For real numbers $a$, $b$, $t_1$, $t_2$ in the following range $b>0, \frac{1}{2}<a+b<1, 0<t_1<1$ and $0<t_2<t_1$, we have
\begin{eqnarray*}
    &&t_1^{2a+2b-1}(1-t_2^2)^{a+b-1}\pFq{}{1}{2a+2b-1;b,b}{a+2b}{\frac{(1-t_1)(t_2+1)}{2(t_2-t_1)},\frac{(1-t_1)(t_2-1)}{2(t_2+t_1)}}\\
    &=&(t_1^2-t_2^2)^{a+b-1}\pFq{}{1}{a+b;a+b-\frac{1}{2},1-a}{a+2b}{\frac{t_1^2-1}{t_1^2},\frac{t_1^2-1}{t_1^2-t_2^2}}.
\end{eqnarray*}
\end{thm}

\begin{rmk}
    We consider another parametrization of this quadratic case.
    Let $0<t_{1}<1<t_{2}$ and let the region $\sigma$ in $\RR^{2}$ defined by
    \[
    \sigma=\{(x,y)\in \RR^2|\  t_{1}<x<t_{2}^{2},1<y<\sqrt{x}\}=\{(x,y)\in \RR^2|\ 1<y<t_{2}, y^2<x<t_{2}^{2} \}.
    \]
    Consider the two-form 
    \[
    \omega=\frac{dx\wedge dy}{(x-y^2)^{2-a-b}(t_{2}^{2}-x)^{b}(x-t_{1})^{a}(y-1)^{2a+2b-2}}.
    \]
    Using Fubini formula, we can obtain the following equation:
    \[
    \int_{1}^{t_{2}}\eta=c\cdot\int_{1}^{t_{2}^{2}}\omega,
    \]
    where 
    \[
    \eta=(t_{2}^{2}-t_{1})^{1-a-b}\frac{dz}{(z-1)^{2a+2b-2}(t_{2}^{2}-z^2)^{1-a}(z^2-t_{1})^{1-b}},
    \]
    \[
    \omega=\frac{dz}{(z-t_{1})^{a}(t_{2}^{2}-z)^{b}(z-1)^{a+b-1}z^{\frac{3}{2}-a-b}}
    \]
    and 
    \[c=2^{2a+2b-3}\frac{\Gamma(3-2a-2b)\Gamma(a)}{\Gamma(2-a-b)\Gamma(1-b)}.\]
    This coincides with the formula \cite[13.9 Page 134]{deligne1993commensurabilities} and the constant $c$ can be read directly by Lemma \ref{lemma: 3 poles}.
\end{rmk}

\subsection{Degeneration to Figure \ref{Figure: (1,2)+(1,0)+(1,0)+(0,1), Degeneration 1}}

If we take $t_2\rightarrow 0$, we obtain
\begin{eqnarray}
\label{equation:quadracase,dege1}
    t_1\pFq{}{1}{2a+2b-1;b,b}{a+2b}{\frac{t_1-1}{2t_1},\frac{t_1-1}{2t_1}}=\pFq{}{1}{a+b;1-a,a+b-\frac{1}{2}}{a+2b}{\frac{t_1^2-1}{t_1^2},\frac{t_1^2-1}{t_1^2}}.
\end{eqnarray}

\begin{figure}[H] %H为当前位置, !htb为忽略美学标准, htbp为浮动图形
\centering %图片居中
\includegraphics[width=0.3\textwidth]{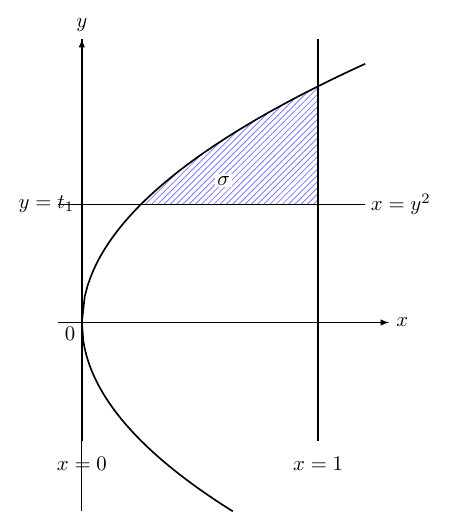} %插入图片, []中设置图片大小, {}中是图片文件名
\caption{(1,2)+(1,0)+(1,0)+(0,1), Degeneration 1} %最终文档中希望显示的图片标题
\label{Figure: (1,2)+(1,0)+(1,0)+(0,1), Degeneration 1}
\end{figure}

The following lemma relates a degeneration of the Appell hypergeometric function and the classical hypergeometric function.

\begin{lem}
For complex numbers $a,b,b^{\prime},c$ with $\mathrm{Re}(c)>\mathrm{Re}(a)>0$ and $x\in \mathbb{C}\setminus [1,\infty)$
\begin{eqnarray*}
    \pFq{}{1}{a;b,b^{\prime}}{c}{x,x}=\pFq{2}{1}{a,b+b^{\prime}}{c}{x}
\end{eqnarray*}
\end{lem}

%By definition, both sides are equal 
%\begin{eqnarray*}\frac{\Gamma(c)}{\Gamma(a)\Gamma(c-a)}\int_{0}^{1}t^{a-1}(1-t)^{c-a-1}(1-xt)^{-2b}dt.\end{eqnarray*}

Applying this lemma to (\ref{equation:quadracase,dege1}), we obtain:
\begin{eqnarray*}
    t\cdot \pFq{2}{1}{2a+2b-1,2b}{a+2b}{{t-1\over 2t}}= \pFq{2}{1}{a+b,b+\frac{1}{2}}{a+2b}{\frac{t^2-1}{t^2}},
\end{eqnarray*}
for $b>0$, $\frac{1}{2}<a+b<1$ and $0<t<1$.

Taking $z=1-{1\over t^2}$, this coincides with Goursat's quadratic transformation \cite{goursat1881equation}, see also \cite[Page 111, (13)]{bateman1953higher}:

\begin{cor}[Goursat's quadratic transformation]
\label{Goursat1}
For real numbers $a,b,z$ in the following range $b<1$ in the following range $\frac{1}{2}<a<b+\frac{1}{2}$ and $z<0$, the Gauss hypergeometric function admits the transformation formula:
    \begin{eqnarray*}
    \pFq{2}{1}{a,b}{a+b-{1\over 2}}{z}=(1-z)^{-{1\over 2}} \pFq{2}{1}{2a-1, 2b-1}{a+b-{1\over 2}}{ {1\over 2}-{1\over 2} \sqrt{1-z}}.
\end{eqnarray*}
\end{cor}

\subsection{Degeneration to Figure \ref{figure: (1,2)+(1,0)+(1,0)+(0,1), Degeneration 2}}
Consider another kind of degeneration by taking $t_2\rightarrow t_1$. Replacing $t_1$ by $t$, we obtain:
\begin{eqnarray*}
    \pFq{2}{1}{2a+b-1,a+b-\frac{1}{2}}{2a+2b-1}{\frac{t^2-1}{t^2}}=t^{2a+b-1}\pFq{2}{1}{2a+b-1,b}{a+b}{-\frac{(1-t)^2}{4t}},
\end{eqnarray*}
for $a,b>0$, $\frac{1}{2}<a+b<1$ and $0<t<1$.
%\begin{eqnarray*}&&\lim_{t_2\rightarrow t_1^2}\frac{\Gamma(1-a-b)\Gamma(b)\Gamma(2a+2b-1)}{2^{2a+2b-1}\Gamma(a+2b)}(1-t_1^2)^{a+2b-1}(1-t_2)^{a+b-1}(t_1^2-t_2)^{-b}\cdot\\&&F_1(2a+2b-1;b,b;a+2b;\frac{(1-t_1)(\sqrt{t_2}+1)}{2(\sqrt{t_2}-t_1)},\frac{(1-t_1)(\sqrt{t_2}-1)}{2(\sqrt{t_2}+t_1)})\\&=&B(1-a-b,b)(1-t_1^2)^{a+b-1}\int_{t_1}^1\frac{dy}{(y^2-t_1^2)^{2-2a-b}(1-y^2)^a}\\&=&\frac{\Gamma(1-a-b)\Gamma(b)\Gamma(2a+b-1)}{2^{2a+2b-1}\Gamma(a+b)}t_1^{-b}(1-t_1^2)^{2a+2b-2}\pFq{2}{1}{b,2a+b-1}{a+b}{-\frac{(1-t_1)^2}{4t_1}}\end{eqnarray*}

%and 

%\begin{eqnarray*}&&\lim_{t_2\rightarrow t_1^2}\frac{\Gamma(b)\Gamma(2a+2b-1)\Gamma(1-a-b)}{2^{2a+2b-1}\Gamma(a+2b)}t_1^{1-2a-2b}(t_1^2-t_2)^{a-1}(1-t_1^2)^{a+2b-1}\cdot\\&&F_1(a+b;a+b-\frac{1}{2},1-a;a+2b;\frac{t_1^2-1}{t_1^2},\frac{t_1^2-1}{t_1^2-t_2})\\&=&\frac{\Gamma(2a+2b-1)\Gamma(1-a-b)}{2^{2a+2b-1}\Gamma(a+b)}\int_{t_1^2}^{1}\frac{dx}{(x-t_1^2)^{2-2a-2b}(1-x)^{1-b}x^{a+b-\frac{1}{2}}}\\&=&\frac{\Gamma(2a+b-1)\Gamma(b)\Gamma(1-a-b)}{2^{2a+2b-1}\Gamma(a+b)}(1-t_1^2)^{2a+2b-2}t_1^{1-2a-2b}\pFq{2}{1}{a+b-\frac{1}{2},2a+b-1}{2a+2b-1}{\frac{t_1^2-1}{t_1^2}}.\end{eqnarray*}

\begin{figure}[H] %H为当前位置, !htb为忽略美学标准, htbp为浮动图形
\centering %图片居中
\includegraphics[width=0.3\textwidth]{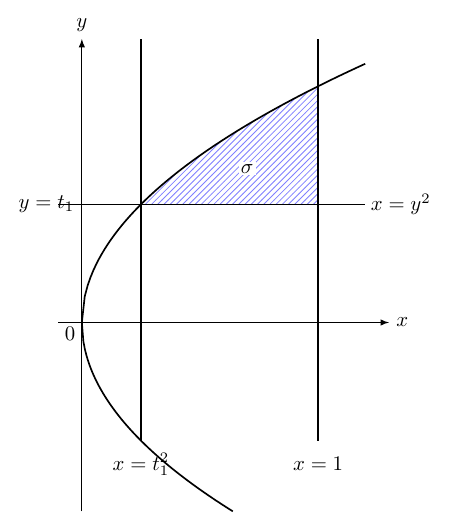} %插入图片, []中设置图片大小, {}中是图片文件名
\caption{(1,2)+(1,0)+(1,0)+(0,1), Degeneration 2} %最终文档中希望显示的图片标题
\label{figure: (1,2)+(1,0)+(1,0)+(0,1), Degeneration 2}
\end{figure}

By taking $z=1-\frac{1}{t^2}$, this coincides with Goursat's quadratic transformation \cite{goursat1881equation}, see also \cite[Page 113, (30)]{bateman1953higher}:

\begin{cor}[Goursat's quadratic transformation]
\label{Goursat2}
For real numbers $a,b,z$ in the following range $0<b<\frac{1}{2}$, $a<2b$ and $z<0$, we have
    \begin{eqnarray*}
    \pFq{2}{1}{a,b}{2b}{z}=(1-z)^{-\frac{a}{2}}\pFq{2}{1}{a,2b-a}{b+\frac{1}{2}}{-\frac{1}{4}(1-z)^{-\frac{1}{2}}(1-(1-z)^{\frac{1}{2}})^2}.
\end{eqnarray*}
\end{cor}

\section{Partition $(2,2)+(1,0)+(0,1)$}
\label{section: (2,2)+(1,0)+(0,1)}

In this section, we consider the case where the divisor $D$ decomposes into three components $D_{1},D_{2},D_{3}$ with degree $(2,2)$, $(1,0)$ and $(0,1)$ respectively.

Let $a,p,q$ be real numbers that satisfy $\frac{1}{2}<a<1$ and $1<p<q$. 

Let $x_{1}<x_{2}$ be two real roots of $-2px^2+(1-pq)x+2q=0$ and $y_{1}<y_{2}$ be two real roots of $-2qy^{2}+(1-pq)y+2p=0$.

Consider the functions
\begin{eqnarray*}
    f(x,y)&=&(qy-1)^{2}+x(-2p+2qy^{2})+x^{2}(y+p)^{2}\\
          &=&(px-1)^{2}+y(-2q+2px^{2})+y^{2}(x+q)^{2}.
\end{eqnarray*}

For $0<x<x_{2}$ fixed, denote by $y_{-}(x)<y_{+}(x)$ two roots of $f(x,y)=0$. For $0<y<y_{2}$ fixed, denote $x_{-}(y)<x_{+}(y)$ two roots of $f(x,y)=0$.

Let $s,t$ be two real numbers that satisfy $\frac{1}{q}<s<y_{2}$ and $0<t<x_{-}(s)$.

We define the divisors as follows:
\begin{eqnarray*}
    D_{1} &:& f(x,y)=0\\
    D_{2} &:& x-t=0\\
    D_{3} &:& y-s=0.
\end{eqnarray*}

We consider the period integral of the form:

\begin{eqnarray*}
    \omega=\frac{dx\wedge dy}{(-f(x,y))^{a}(t-x)^{2-2a}(s-y)^{2-2a}}.
\end{eqnarray*}

\begin{figure}[H] %H为当前位置, !htb为忽略美学标准, htbp为浮动图形
\centering %图片居中
\includegraphics[width=0.35\textwidth]{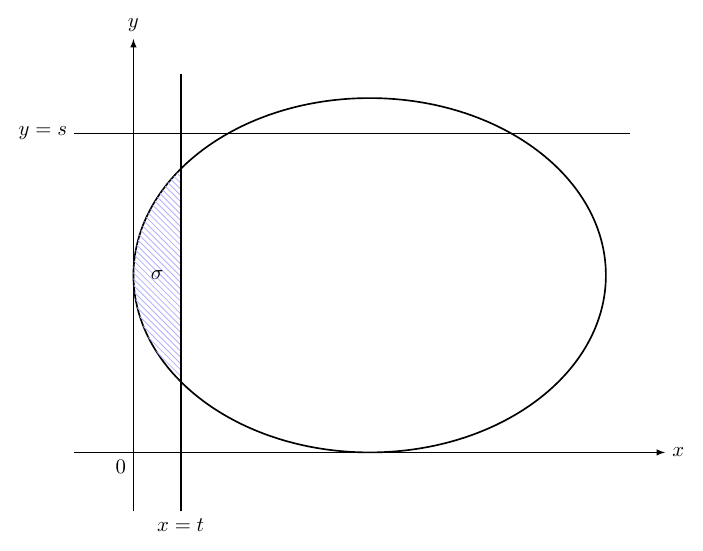} %插入图片, []中设置图片大小, {}中是图片文件名
\caption{(2,2)+(1,0)+(0,1) Segment Region} %最终文档中希望显示的图片标题
\label{figure: (2,2)+(1,0)+(0,1) Segment Region}
\end{figure}

The specified integration cycle $\sigma$ (see Figure \ref{figure: (2,2)+(1,0)+(0,1) Segment Region}) is defined by 
\begin{eqnarray*}
    \sigma &=& \{(x,y)\in \RR^{2}\ | \ 0<x<t, y_{-}(x)<y<y_{+}(x) \}\\
           &=& \{(x,y)\in \RR^{2}\ | \ y_{-}(t)<y<y_{+}(t), x_{-}(y)<x<t \}.
\end{eqnarray*}

Using a similar argument as in $\S$ \ref{section:partition(1,2)(1,1)(1,0)}, we obtain the following transformation formula relating Appell--Lauricella hypergeometric function $F_{D}^{(4)}$.
\begin{thm}
    Let $a,p,q,s,t$ be real numbers in the range that $\frac{1}{2}<a<1$, $1<p<q$, $\frac{1}{q}<s<y_{2}$ and $0<t<x_{-}(s)$. We have the following transformation formula relating $F_{D}^{(4)}$:
\begin{eqnarray*}
    &&\Fd{(4)}{D}{\frac{3}{2}-a;a-\frac{1}{2},a-\frac{1}{2},1-a,1-a}{\frac{1}{2}+a}{\frac{t}{x_{1}},\frac{t}{x_{2}},\frac{t}{x_{-}(s)},\frac{t}{x_{-}(s)}}\\
    &=&C\cdot\Fd{(4)}{D}{a;2-2a,a-\frac{1}{2},a-\frac{1}{2},a-\frac{1}{2}}{2a}{\frac{\delta}{s-y_{-}(t)},-\frac{\delta}{y_{-}(t)},\frac{\delta}{y_{2}-y_{-}(t)},\frac{\delta}{y_{1}-y_{-}(t)}},
\end{eqnarray*}
where 
\[\delta=(4(pq-1)t(t+q)^{-4}(-2pt^2+(1-pq)t+2q))^{\frac{1}{2}}\] 
and 
\[C=\frac{2^{1-2a}}{(t+q)^{2a}}(\frac{s-y_{-}(t)}{qs-1})^{2a-2}(\frac{(pq-1)(-2pt^2+(1-pq)t+2q)}{y_{-}(t)(-y_{-}(t)^2+\frac{1-pq}{2q}y_{-}(t)+\frac{p}{q})})^{a-\frac{1}{2}}.\]
\end{thm}

\section{Partition $(1,2)+(2,1)$}
\label{section: (1,2)+(2,1)}
In this section, we consider the case $\deg D_1=(1,2)$ and $\deg D_2=(2,1)$.

Let $a,b,c>0$ be real parameters satisfying $b^2-4ac>0$. Let $t>0$ be a real parameter satisfying $t>(a+b+c)^2$. Let the divisors be given by
\begin{eqnarray*}
    D_{1} &:& y(ax^{2}+bx+c)-x^{2}=0\\
    D_{2} &:& x-ty^{2}=0.
\end{eqnarray*}

Let $r_{+}(y)>r_{-}(y)$ be two real roots of the quadratic equation $(1-ay)x^2-byx-cy=0$ with respect to $x$. Consider the equation 
\begin{eqnarray*}
at^2y^4-t^2y^3+bty+c=0,
\end{eqnarray*}
which admits $y_{0}$ as its smallest positive real root. Let $y_{1},y_{2},y_{3}$ be the other three roots.

\begin{figure}[H] %H为当前位置, !htb为忽略美学标准, htbp为浮动图形
\centering %图片居中
\includegraphics[width=0.35\textwidth]{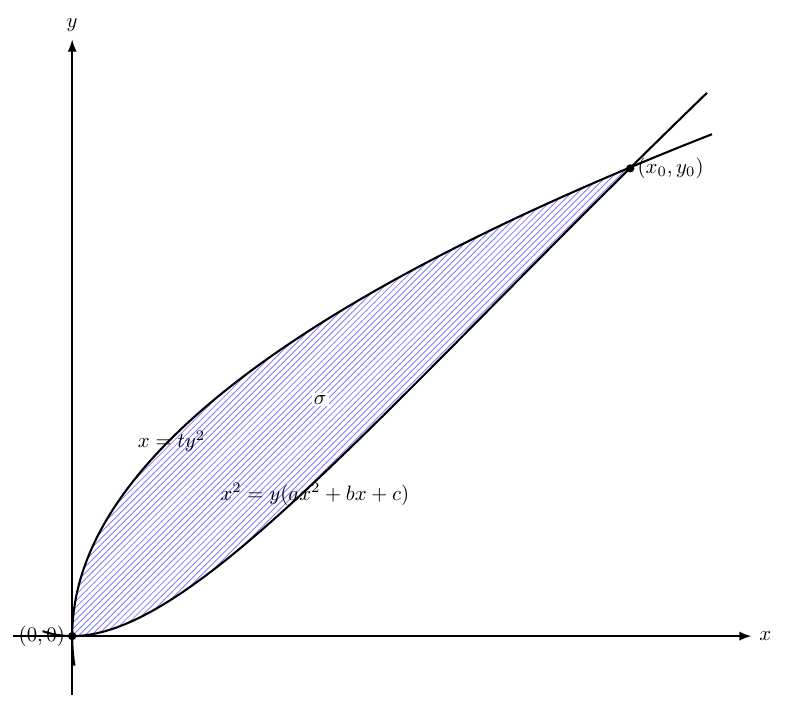} %插入图片, []中设置图片大小, {}中是图片文件名
\caption{(1,2)+(2,1) Region} %最终文档中希望显示的图片标题
\label{figure: (1,2)+(2,1) Region}
\end{figure}

Let $\sigma$ be a region (see Figure \ref{figure: (1,2)+(2,1) Region}) in $\RR^{2}$ defined by
\begin{eqnarray*}
    \sigma &=&\{ (x,y)\in \RR^2|\ 0<y<y_{0},\ ty^2<x<r_{+}(y)\}\\
    &=&\{(x,y)\in \RR^2|\ 0<x<ty_{0}^2, \ \frac{x^2}{ax^2+bx+c}<y<\sqrt{\frac{x}{t}}\}.
\end{eqnarray*}

Let $\omega$ be a two-form
\begin{eqnarray*}
    \omega=\frac{dx\wedge dy}{(y(ax^2+bx+c)-x^{2})^{\frac{2}{3}}(x-ty^2)^{\frac{2}{3}}}.
\end{eqnarray*}

Using a similar argument as in $\S$ \ref{section:partition(1,2)(1,1)(1,0)}, we obtain the following transformation formula relating Appell--Lauricella hypergeometric function $F_{D}^{(3)}$.

\begin{thm}
Let $a,b,c,t$ be real numbers in the range $a,b,c,t>0$, $b^2-4ac>0$ and $t>(a+b+c)^{2}$. Let $y_0,y_1,y_2,y_3$ be the four roots of $at^2y^4-t^2y^3+bty+c$. Let $y_0$ be the smallest positive real root. We have
\begin{eqnarray*}
    &&c^{\frac{1}{6}}\Fd{(3)}{D}{\frac{1}{2};\frac{1}{3},\frac{1}{3},\frac{1}{6}}{\frac{7}{6}}{\frac{y_{0}(y_2-y_1)}{y_2(y_0-y_1)},\frac{y_0(y_3-y_1)}{y_3(y_0-y_1)},\frac{y_0((b^2-4ac)y_1+4c)}{4c(y_0-y_1)}}\\
    &=&(\frac{y_0y_1}{y_0+y_1})^{\frac{1}{2}}t^{\frac{1}{3}}\Fd{(3)}{D}{\frac{1}{2};\frac{1}{3},\frac{1}{3},\frac{1}{6}}{\frac{7}{6}}{\frac{y_0^2(y_2^2-y_1^2)}{y_2^2(y_0^2-y_1^2)},\frac{y_0^2(y_3^{2}-y_1^2)}{y_3^2(y_0^2-y_1^2)},\frac{y_0^2}{y_0^2-y_1^2}}.
\end{eqnarray*}
\end{thm}

\section{Partition $(1,3)+(1,0)+(1,0)$}
\label{section:(1,3)+(1,0)+(1,0)}
In this section, we consider the case $\deg D_1=(1,3)$, $\deg D_2=(1,0)$ and $\deg D_3=(1,0)$.

\subsection{Period integral in Figure \ref{figure: (3,1)+(0,1)+(0,1) Segment Region}}

Let $t,a,s,w$ be real numbers satisfying $-\frac{3}{2}<t<-1$, $a<1$ and $0<s<w<1$. We consider the following cubic polynomial of $y$:
\[
f(x,y)=y^3+ty^2-(3+2t)xy+(2+t)x,
\]
which has discriminant
\[
\Delta(x)=4x(x-1)((3+2t)^3x+(t^4+2t^3)).
\]
Consider the divisors:
\begin{eqnarray*}
    D_{1} &:& f(x,y)=0\\
    D_{2} &:& x-s=0\\
    D_{3} &:& x-w=0.
\end{eqnarray*}
For any fixed $0<x<1$, the cubic polynomial $f(x,y)$ of $y$ has 3 real roots $r_1(x)<r_2(x)<r_3(x)$. On the other hand, fixing $y\in \RR\setminus\{\frac{2+t}{3+2t}\}$, let $F_t(y)=\frac{y^3+ty^2}{(3+2t)y-(2+t)}$ be the root of $f(x,y)$, viewed as a degree-one polynomial of $x$.

For $x\neq r_{1}(s)$, let $R(x)=\frac{(r_{2}(s)-r_{1}(s))(x-r_{3}(s))}{(r_{2}(s)-r_{3}(s))(x-r_{1}(s))}$. Let $x_{0}=-\frac{t^{3}(t+2)}{(2t+3)^{3}}$.

\begin{figure}[H] %H为当前位置, !htb为忽略美学标准, htbp为浮动图形
\centering %图片居中
\includegraphics[width=0.3\textwidth]{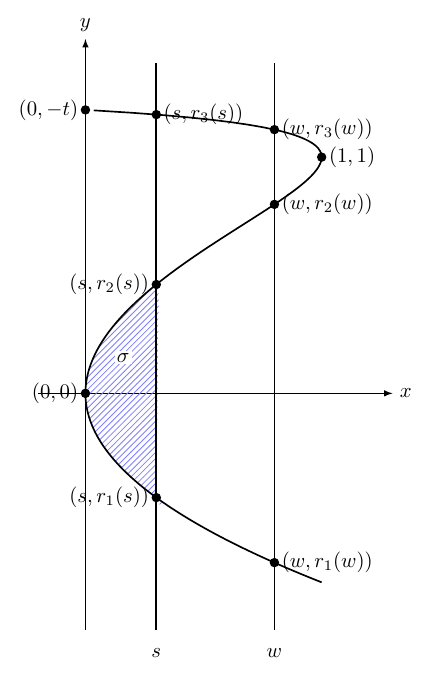} %插入图片, []中设置图片大小, {}中是图片文件名
\caption{(3,1)+(0,1)+(0,1) Segment Region} %最终文档中希望显示的图片标题
\label{figure: (3,1)+(0,1)+(0,1) Segment Region}
\end{figure}

Consider the following cycle $\sigma\subset \RR^2$
\begin{eqnarray*}
    \sigma&=&\{(x,y)\in \RR^{2}\ | \ 0<x<s, \ r_1(x)<y<r_2(x)\}\\
              &=&\{(x,y)\in \RR^{2}\ | r_{1}(s)<y<r_{2}(s), \ F_{t}(y)<x<s\}
\end{eqnarray*}
and a two-form $\omega$
\begin{eqnarray*}
    \omega&=&\frac{dx\wedge dy}{f(x,y)^{\frac{2}{3}}(s-x)^{a}(w-x)^{\frac{4}{3}-a}}\\
    &=&\frac{dx\wedge dy}{(y^3+ty^2-(3+2t)xy+(2+t)x)^{\frac{2}{3}}(s-x)^{a}(w-x)^{\frac{4}{3}-a}}.
\end{eqnarray*}

We obtain the following transformation formula.

\begin{thm}
    Let $t,a,s,w$ be real numbers satisfying $-\frac{3}{2}<t<-1$, $a<1$ and $0<s<w$. We have
\begin{eqnarray*}
    &&\Fd{(3)}{D}{\frac{5}{6};\frac{1}{6},\frac{1}{6},\frac{4}{3}-a}{\frac{11}{6}-a}{s,\frac{s}{x_{0}},\frac{s}{w}}\\
    &=&C\cdot \Fd{(3)}{D}{\frac{4}{3}-a;1-a,1-a,1-a}{\frac{8}{3}-2a}{R(r_{1}(w)),R(r_{2}(w)),R(r_{3}(w))},
\end{eqnarray*}
where 
\[C=\frac{(-t)^{\frac{1}{2}}(t+2)^{\frac{1}{6}}w^{\frac{4}{3}-a}((s-1)((3+2t)^{3}s+t^{4}+2t^{3}))^{\frac{5}{6}-a}}{(w-s)^{\frac{1}{3}}(r_{3}(s)-r_{2}(s))^{3-3a}f(w,r_{1}(s))^{1-a}}.\]
\end{thm}

\subsection{Degeneration to Figure \ref{figure: (3,1)+(0,1)+(0,1), Degeneration}}
If we take $s\rightarrow 1$ and $w\rightarrow x_{0}$, then the form degenerates to:
\begin{eqnarray*}
    \omega&=&\frac{dx\wedge dy}{f(x,y)^{\frac{2}{3}}(1-x)^{a}(x_{0}-x)^{\frac{4}{3}-a}}.
\end{eqnarray*}

\begin{figure}[H] %H为当前位置, !htb为忽略美学标准, htbp为浮动图形
\centering %图片居中
\includegraphics[width=0.3\textwidth]{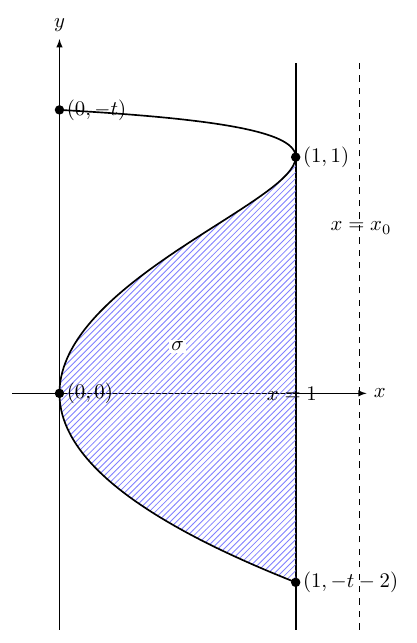} %插入图片, []中设置图片大小, {}中是图片文件名
\caption{(3,1)+(0,1)+(0,1), Degeneration} %最终文档中希望显示的图片标题
\label{figure: (3,1)+(0,1)+(0,1), Degeneration}
\end{figure}

Then the cycle $\sigma$ (see Figure \ref{figure: (3,1)+(0,1)+(0,1), Degeneration}) becomes
\begin{eqnarray*}
    \sigma&=&\{(x,y)\in \RR^{2}\ | \ 0<x<1, \ r_1(x)<y<r_2(x)\}\\
              &=&\{(x,y)\in \RR^{2}\ | -t-2<y<1, \ F_{t}(y)<x<1\}.
\end{eqnarray*}

We obtain the following transformation formula of Gauss hypergeometric function:

\begin{cor}
\label{not known yet}
     Let $t,a$ be real numbers in the following range $-\frac{3}{2}<t<-1$ and $a<\frac{5}{6}$, we have
\begin{eqnarray*}
    \pFq{2}{1}{\frac{4}{3}-a,2-2a}{3-3a}{-\frac{t(2t+3)}{t+2}}
    = C\cdot \pFq{2}{1}{\frac{5}{6},\frac{3}{2}-a}{\frac{5}{3}-a}{-\frac{(2t+3)^{3}}{t^{3}(t+2)}},
\end{eqnarray*}
where $C=2^{2a-2}3^{\frac{5}{2}-3a}(-t)^{3a-\frac{9}{2}}(t+2)^{\frac{1}{2}-a}(t+3)^{2}$.
\end{cor}

\bibliography{reference}
\bibstyle{alpha}
	
\Addresses

\end{document}